\documentclass[a4paper,UKenglish]{lipics-v2016} 
\usepackage{microtype}
\usepackage{mathrsfs}
\usepackage{lineno, blindtext}
\title{Two Phase Transitions in Two-way Bootstrap Percolation}
\titlerunning{Two Phase Transitions in Two-way Bootstrap Percolation}
\author[1]{Ahad N. Zehmakan}
\affil[1]{ETH Zurich, Switzerland\\
\texttt{abdolahad.noori@inf.ethz.ch}}
\authorrunning{Ahad N. Zehmakan} 
\Copyright{}
\EventEditors{}
\EventNoEds{2}
\EventLongTitle{.}
\EventShortTitle{}
\EventAcronym{}
\EventYear{2019}
\EventDate{}
\EventLocation{}
\EventLogo{}
\SeriesVolume{}
\ArticleNo{}

\begin{document}
\maketitle

\begin{abstract}
Consider a graph $G$ and an initial random configuration, where each node is black with probability $p$ and white otherwise, independently. In discrete-time rounds, each node becomes black if it has at least $r$ black neighbors and white otherwise. We prove that this basic process exhibits a threshold behavior with two phase transitions when the underlying graph is a $d$-dimensional torus and identify the threshold values.
\end{abstract}
\vspace{-0.1 cm}
\section{Introduction}
Consider a graph $G=(V,E)$ and an initial \emph{random configuration}, where each node is independently black with probability $p$ and white otherwise. In \emph{$r$-bootstrap percolation} (or shortly $r$-BP) for some positive integer $r$, in each discrete-time round white nodes with at least $r$ black neighbors become black and black nodes stay unchanged. This basic process is meant to model different progressive dynamics like rumor spreading in a society, fire propagation in a forest, and infection spreading among cells, where a black/white node corresponds to an individual who is informed/uninformed of a rumor, a tree which is or not on fire, or an infected/uninfected cell.

In the above examples if a node becomes black, it remains black forever. For instance, an individual who is informed of a rumor remains informed. However, there exist many real-world examples where nodes might keep switching between black and white. For example, two service providers might be competing to get people adopting their services, and thus users may switch among two services back and forth. Another example is opinion forming regarding an election in a community, where an individual might adopt the positive opinion if a certain number/fraction of its connections are positive, and become negative otherwise. To study this kind of non-progressive processes the following model has been introduced. For a graph $G$ and an initial random configuration, in \emph{two-way $r$-bootstrap percolation} in each round a node becomes black if it has at least $r$ black neighbors and white otherwise.

The behavior of these two basic models have been extensively studied by researchers from a wide spectrum of fields, like statistical physics~\cite{aizenman1988metastability,schonmann1992behavior}, distributed computing~\cite{peleg1997local,flocchini2004dynamic,frischknecht2013convergence}, mathematics~\cite{balister2010random,kanoria2011majority}, and even sociology~\cite{granovetter1978threshold} due to their various applications, such as distributed fault-local mending~\cite{peleg1997local}, modeling biological interactions~\cite{molofsky1999local}, viral marketing~\cite{kempe2003maximizing}, and modeling disordered magnetic systems~\cite{balogh2012sharp}.

The first natural question arises: How long does it take for the above processes to stabilize? Both of these processes are induced by deterministic updating rules and for a graph $G=(V,E)$ there are $2^{|V|}$ possible colorings (configurations). Therefore, by starting from whatever initial configuration the process must reach a cycle of configurations eventually. The length of this cycle and the number of rounds the process needs to reach the cycle are respectively called the \emph{period} and the \emph{consensus time} of the process. In $r$-BP, the periodicity is always one and the consensus time is bounded by $|V|-1$, which is tight (for instance consider a path $P_n$ and $1$-BP, where initially all nodes are white except one of the leaves). In two-way $r$-BP, $2^{|V|}$ is a trivial upper bound on both period and consensus time. However, interestingly Goles and Olivos~\cite{goles1980periodic} proved that the period is always one or two and Fogelman, Goles, and Weisbuch~\cite{fogelman1983transient} showed that the consensus time is bounded by $\mathcal{O}(|E|)$, which is tight (consider a cycle $C_n$ which is fully white except two adjacent black nodes for $r=1$).

Arguably, the most well-studied question concerning the behavior of these models is: What is the minimum $p$ for which the process becomes fully black with probability approaching one? This question has been investigated on different classes of graphs like hypercube~\cite{balogh2006bootstrap}, the binomial random graph~\cite{janson2012bootstrap,chang2013bounding}, random regular graphs~\cite{balogh2007bootstrap,gartner2017majority,mossel2014majority}, infinite trees~\cite{kanoria2011majority}, and many others. A substantial amount of attention has been devoted to address this question on the $d$-dimensional torus, due to the study of certain interacting particle systems like fluid flow in rocks~\cite{adler1988diffusion}, dynamics of glasses~\cite{garrahan2011kinetically}, and biological interactions~\cite{molofsky1999local}. The $d-$dimensional torus $\mathbb{T}_L^d$ is the graph with node set $[L]^d:=\{1,\cdots,L\}^d$, where two nodes are adjacent if and only if they differ by $1$ or $L-1$ in exactly one coordinate. Notice we always assume that $d$ is a constant while let $L$ tend to infinity.

The aforementioned question regarding $r$-BP on the $d$-dimensional torus $\mathbb{T}_L^d$ first was considered by Aizenman and Lebowitz~\cite{aizenman1988metastability}, who proved that for $r=d=2$ the process exhibits a \emph{weak threshold behavior} at $\mathscr{P}_1:=(\log_{(r-1)}L)^{-(d-r+1)}$ where $\log_{(r)}L:=\log \log_{(r-1)}L$ for $r\geq 1$ and $\log_0 L=L$. That is, the process becomes fully black for $p\gg \mathscr{P}_1$ and it does not for $p\ll \mathscr{P}_1$ asymptotically almost surely~\footnote{For a graph $G=(V,E)$ we say an event happens asymptotically
almost surely (a.a.s.) if it happens with probability $1-o(1)$ as $|V|$ tends to infinity.}, where we shortly write $f\ll g$ instead of $f=o(g)$ for two functions $f(L),g(L)$. Cerf and Cirillo~\cite{cerf1999finite} made one step further by proving that the process exhibits a similar weak threshold behavior at $\mathscr{P}_1$ for $d=r=3$. Finally, Cerf and Manzo~\cite{cerf2002threshold} extended this result to all values of $1\le r\le d$, building on the work by Schonmann~\cite{schonmann1992behavior}. Later on, it was proven, by Holroyd~\cite{holroyd2003sharp}, that for $d=r=2$ the process actually exhibits a \emph{sharp threshold behavior}; that is, the process a.a.s. becomes fully black if $p\ge (1+\epsilon)\lambda\mathscr{P}_1$ and does not if $p\le (1-\epsilon)\lambda\mathscr{P}_1$  for any constant $\epsilon>0$, where $\lambda(d,r)>0$ is a constant. This sharp threshold behavior was proven for the case of $d=r=3$ by Balogh, Bollobas, and Morris~\cite{balogh2009bootstrap} and finally for all values of $1\le r\le d$ by Balogh, Bollobas, Duminil-Copin, and Morris~\cite{balogh2012sharp}. Along the way, as an intermediate step the behavior of a similar process was also studied. In \emph{modified $r$-bootstrap percolation} on $\mathbb{T}_L^d$, by starting from a random initial configuration in every round each white node becomes black if it has black neighbor(s) in at least $r$ distinct dimensions and black nodes remain unchanged. See~\cite{holroyd2006metastability,holroyd2003sharp} by Holroyd regarding the sharp threshold behavior of modified $r$-BP for $r=d$.

For two-way $r$-BP on $\mathbb{T}_L^d$ Schonmann~\cite{schonmann1990finite}, by applying the results from~\cite{aizenman1988metastability}, proved that for $d=r=2$ the process becomes fully black if $p\gg 1/\sqrt{\log L}$ and it does not if $p\ll 1/\sqrt{\log L}$ a.a.s., i.e., it exhibits a weak threshold behavior. What about the higher dimensions? Despite several attempts~\cite{coker2014sharp,molofsky1999local,balister2010random,schonmann1990finite} over the last three decades, this question has remained open. Intuitively speaking, the inherent difficulty of analyzing two-way $r$-BP comes from the fact that unlike $r$-BP, in two-way $r$-BP a node may switch between two colors back and forth.

By providing several new techniques, using some ideas inspired from \cite{schonmann1990finite,gartner2017biased} (where the special case of $r=d=2$ is handled), and applying some prior results regarding (modified) $r$-BP~\cite{cerf2002threshold,balogh2012sharp,holroyd2006metastability}, we extend the above threshold behavior in two-way $r$-BP to all dimensions. 

One might relax the above question and asks what is the minimum $p$ for which black color survives (but does not make the whole graph black necessarily). In $r$-BP on $\mathbb{T}_L^d$ as will be discussed, it is straightforward to show that black color survives forever if $p\gg \mathscr{P}_2$ and it does not if $p\ll \mathscr{P}_2$ a.a.s. for $\mathscr{P}_2:=L^{-d}$. The answer to this question for two-way $r$-BP is somewhat more involved. We prove a similar threshold behavior in the two-way setting, which leads into some interesting insights regarding the behavior of the process.

All in all, we prove two-way $r$-BP on the $d$-dimensional torus $\mathbb{T}_L^d$ exhibits two phase transitions. More precisely, asymptotically almost surely 
\begin{itemize}
\item the process becomes fully white if $p\ll \mathscr{P}_2^{1/2^{r-1}}$: Phase 1
\item both colors survive if $\mathscr{P}_2^{1/2^{r-1}} \ll p \ll \mathscr{P}_1^{1/2^{r-1}}$: Phase 2
\item the process becomes fully black if $\mathscr{P}_1^{1/2^{r-1}}\ll p$: Phase 3.
\end{itemize}
\vspace{-0.25 cm}
\begin{figure}[h]
\begin{center}
\includegraphics[width=0.5	\textwidth]{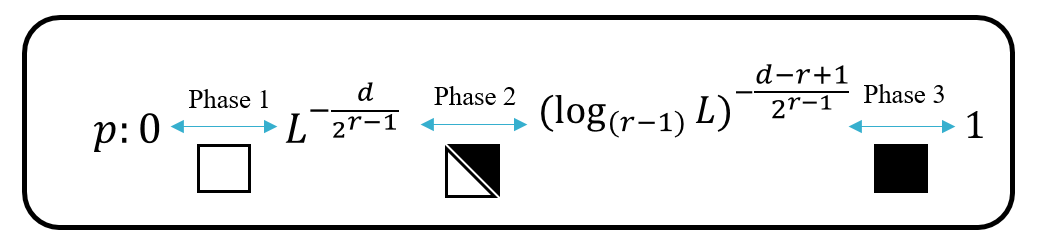}
\caption{\centering Two phase transitions in two-way $r$-BP on $\mathbb{T}_L^d$.\label{fig 1}}
\end{center}
\end{figure}
\vspace{-0.5 cm}

After setting up some basic definitions in Section~\ref{definitions}, we provide some insights and the main ideas behind our proof techniques in Section~\ref{proof technique}.
Finally in Section~\ref{two phase transition}, our main results regarding two phase transitions of two-way $r$-BP on the $d$-dimensional torus are provided.
\vspace{-0.25 cm}
\subsection{Definitions and Preliminaries}
\vspace{-0.1 cm}
\label{definitions}
Let for a graph $G=(V,E)$ and a node $v\in V$, the \textit{neighborhood} of $v$ be $N(v):=\{u\in V: \{v,u\}\in E\}$. For a set $S\subseteq V$ we have $N_S(v):=N(v) \cap S$ and $N(S):=\bigcup_{v\in S}N(v)$. Furthermore, for two nodes $v,u\in V$ we define the \emph{distance} $d(v,u)$ to be the length of a shortest path between $v$ and $u$, in terms of the number of edges. Let $G^2$ be the second power of graph $G$, where two nodes are adjacent if their distance in $G$ is at most 2. Then, we say there is a \emph{semi-connected} path between $v$ and $u$ in $G$ if there is a path between them in $G^2$.

Formally, a \emph{configuration} is a function $\mathcal{C}:V\rightarrow\{b,w\}$, where $b$,$w$ stand for black and white. For a configuration $\mathcal{C}$, node $v\in V$, and color $c\in\{b,w\}$, we define $N^{\mathcal{C}}_c(v) := \{u\in N(v): \mathcal{C}(u) = c\}$ which is the set of neighbors of $v$ which have color $c$ in configuration $\mathcal{C}$. Finally, for a color $c\in \{b,w\}$ and a set $S\subseteq V$, we write $\mathcal{C}|_S=c$ if $\forall v\in S$, $\mathcal{C}(v)=c$. 

Let us define (two-way) $r$-BP formally. Consider a graph $G=(V,E)$ and an initial random configuration $\mathcal{C}_0$. In two-way $r$-BP, $\mathcal{C}_t(v)=b$ if $|N^{\mathcal{C}_{t-1}}_b(v)|\geq r$ and $\mathcal{C}_t(v)=w$ otherwise for $t\geq 1$, where $\mathcal{C}_t$ is the $t$-th configuration. In $r$-BP, $\mathcal{C}_t(v)=b$ if $|N^{\mathcal{C}_{t-1}}_b(v)|\geq r$ or $\mathcal{C}_{t-1}=b$ and $\mathcal{C}_t(v)=w$ otherwise.

For a graph $G$ and two configurations $\mathcal{C}$ and $\mathcal{C}'$ we write $\mathcal{C}\leq \mathcal{C'}$ if all black nodes in $\mathcal{C}$ are also black in $\mathcal{C}'$. A model $M_1$ is \emph{stronger} than model $M_2$ if for any graph $G$ and any configuration $\mathcal{C}$, we have $M_2(\mathcal{C})\leq M_1(\mathcal{C})$ where $M_1(\mathcal{C})$ and $M_2(\mathcal{C})$ denote the configuration obtained from $\mathcal{C}$ after one round of $M_1$ and $M_2$. For instance, $r$-BP is stronger than two-way $r$-BP. Furthermore, $M$ is a \emph{monotone} model if for any graph $G$ and any two configurations $\mathcal{C}_1\leq \mathcal{C}_2$, we have $\mathcal{C}^{\prime}_1\leq \mathcal{C}_{2}^{\prime}$ where $\mathcal{{C}}^{\prime}_1$ and $\mathcal{C}^{\prime}_2$ are the configurations obtained respectively from $\mathcal{C}_1$ and $\mathcal{C}_2$ after one round of $M$. All models introduced in this paper are monotone. 
 
For any model $M$ and a graph $G=(V,E)$, a set $S\subseteq V$ is called a \emph{c-robust set} for $c\in\{b,w\}$ whenever the following holds: if all nodes in $S$ share color $c$ in some configuration during the process, then they will all keep it in all upcoming configurations. Furthermore, a set $S\subseteq V$ is \emph{c-eternal} for $c\in\{b,w\}$ means if all nodes in $S$ have color $c$ in some configuration, then color $c$ \emph{survives}, that is for any upcoming configuration there is a node which has color $c$. Clearly, a $c$-robust set is also a $c$-eternal set, but not necessarily the other way around. Furthermore, a \emph{$c$-dynamo} for $c\in\{b,w\}$ is a subset of nodes which \emph{takes over} if they share color $c$, meaning the whole graph will have color $c$ after some rounds. For example in any connected graph and two-way $1$-BP, any two adjacent nodes are a $b$-dynamo. Notice one node might not suffice, for instance in an even cycle. For some graph

For some graph $G$ and an integer $r\ge 1$, we say a node set $S$ is $(r,c)$-robust (analogously, $(r,c)$-eternal, $(r,c)$-dynamo) if it is $c$-robust (resp. $c$-eternal, $c$-dynamo) in two-way $r$-BP on $G$. Observe that in an $(r,b)$-robust set (analogously $(r,w)$-robust set) for each node $v\in S$, $|N_S(v)|\geq r$ (resp. $|N_{V\setminus S}(v)|<r$). 

The $d$-dimensional torus $\mathbb{T}_L^d$ is the graph with the node set $V=\{(i_1,\cdots,i_d): 1\le i_1,\cdots,i_d\le L)\}$
and the edge set $E= \{\{i,i'\}:|i_j-i^{\prime}_j|=1, L-1 \ \textrm{for some}\  j \ \textrm{and}\ i_k=i^{\prime}_k \ \forall k\ne j\}$. Notice each node in $\mathbb{T}_L^d$ has $2d$ neighbors, two neighbors in each dimension. For a node $v=(i_1,\cdots,i_d)$ and $1\leq j\leq d$, we call $(i_1,\cdots, i_j+1,\cdots,i_d)$ and $(i_1,\cdots, i_j-1,\cdots,i_d)$ the neighbors of $v$ in the $j$-th dimension. (For the above definition to make sense when $i_j$ is equal to $1$ or $L$, we need to apply the modulo $L$ operation. However to lighten the notation, we skip that whenever it is clear from the context.)

The \emph{hyper-rectangle} of size $l_1 \times \cdots \times l_d$ starting from node $(i_1, \cdots, i_d)$ is the node set $\{(i_1^{\prime},\cdots,i_d^{\prime}):i_j\leq i_j^{\prime}\leq i_j+l_j \ \forall 1\leq j\leq d\}$. An $r$-dimensional \emph{hyper-square} $HS$ starting at node $i$ is a hyper-rectangle starting at $i$ with exactly $r$ of $l_j$s being equal to 1 and the rest 0, where we define $J_{HS}:=\{j:l_j\ne 0\}$. We denote the \emph{odd-part} (analogously \emph{even-part}) of $HS$ by $HS^{(1)}$ (resp. $HS^{(2)}$), which are the nodes that differ in odd (resp. even) number of coordinates with $i$. As a warm-up let us prove the following simple, however crucial, lemma.
\vspace{-0.15 cm}
\begin{lemma}
\label{lemma 4}
For an $r$-dimensional hyper-square $HS$ in $\mathbb{T}_L^d=(V,E)$, $HS^{(1)}$ and $HS^{(2)}$ are $(r,b)$-eternal sets.
\end{lemma}
\vspace{-0.4 cm}
\begin{proof}
It suffices to show each node in $HS^{(1)}$ has exactly $r$ neighbors in $HS^{(2)}$ and vice versa because it implies that if $\mathcal{C}_t|_{HS^{(1)}}=b$, we will have $\mathcal{C}_{t+2t'+1}|_{HS^{(2)}}=b$ and $\mathcal{C}_{t+2t'}|_{HS^{(1)}}=b$ for any $t'\geq 0$ (a similar argument for $\mathcal{C}_t|_{HS^{(2)}}=b$). Let $i'$ be a node in $HS^{(1)}$ and assume $HS$ starts at node $i$. There is an odd-size subset $J\subseteq J_{HS}$ of coordinates in which $i'$ is larger than $i$ by one. Now, by decrementing any coordinate in $J$ or incrementing any coordinate in $J_{HS}\setminus J$, we reach a node in $HS^{(2)}$ which is a neighbor of $i'$. Thus, $i'$ has $r$ neighbors in $HS^{(2)}$. The proof of the other direction is analogous. See Figure~\ref{fig 2} for an example.
\end{proof}
\vspace{-0.5 cm}
\begin{figure}[h]
\begin{center}
\includegraphics[width=0.6	\textwidth]{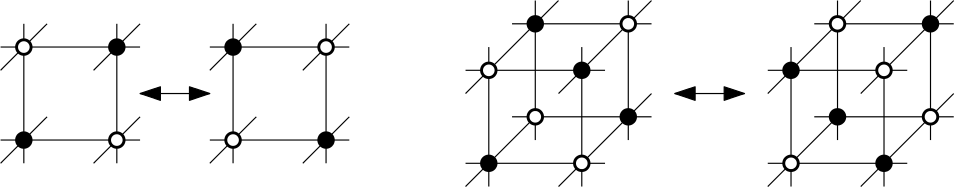}
\caption{\centering(left) A $(2,b)$-eternal set (right) a $(3,b)$-eternal set.\label{fig 2}}
\end{center}
\vspace{-0.5 cm}
\end{figure}
Notice the even/odd-part of an $r$-dimensional hyper-square is of size $2^{r-1}$ which implies that there exists an $(r,b)$-eternal set of size $2^{r-1}$. Furthermore, we always assume $2\leq r\leq d$. The setting of $d+1\leq r\leq 2d$ is the same as $1\leq r\leq d$ if we swap black and white. The case of $r=1$ is trivial since any two adjacent black nodes make the whole graph black.

In the $r$-dimensional torus $\mathbb{T}_L^r$, the set of $r$-dimensional hyper-squares whose starting node is in $\{(2i_1-1,\cdots,2i_r-1):1\leq i_1,\cdots,i_r\leq \lfloor L/2\rfloor\}$ divide the node set (except the nodes with value $L$ in one of their coordinates if $L$ is odd) into $\lfloor L/2\rfloor^r$ pair-wise disjoint hyper-squares. Furthermore, if we divide the nodes in the $d$-dimensional torus $\mathbb{{T}}_L^d$ into $L^{d-r}$ pair-wise disjoint subsets according to their last $d-r$ coordinates, the induced subgraph by each of these subsets is an $r$-dimensional torus. Now, if we partition the node set of each of these $r$-dimensional tori into $\lfloor L/2\rfloor^r$ hyper-squares as above, we will have $L^{d-r}\lfloor L/2\rfloor^r=\Theta(L^d)$ pair-wise disjoint $r$-dimensional hyper-squares. We call this procedure the \emph{tiling} of $\mathbb{T}_L^d$ into $r$-dimensional hyper-squares.
\vspace{-0.3 cm}
\subsection{Proof Techniques and Some Insights}
\vspace{-0.1 cm}
\label{proof technique} 
\textbf{Phase transition.} Intuitively speaking, one might expect any monotone model to exhibit some sort of threshold behavior with two phase transitions on any graph. Assume that the initial probability $p$ is very close to zero, then black nodes probably disappear in a few number of rounds. However, if we gradually increase the initial probability, at some point it would suffice to guarantee the survival of black color, however perhaps it is not high enough to result in a fully black configuration. Finally, if we keep increasing the initial probability, suddenly it should be sufficient to guarantee not only the survival of black color but also the disappearance of white color. Another way of seeing these two phase transitions is in terms of $b$-eternal set and $b$-dynamo. One might think of the first threshold as the threshold value for having a fully black $b$-eternal set and the second one as the threshold for having a fully black $b$-dynamo since black color survives if and only if there is a black $b$-eternal set initially and it will take over if and only if there is a black $b$-dynamo in the initial configuration. Notice the first and the second threshold values might match, which means the process goes actually through one phase transition; for instance, in $1$-BP the existence of a black node is the necessary condition for survival of black color and at the same time sufficient condition to take over the whole graph; i.e., any $b$-eternal set is also a $b$-dynamo. Although this threshold behavior might seem conceptually simple, identifying the exact threshold values is usually a very non-trivial task. As we discussed even for the very special case of $r$-BP on $\mathbb{T}_L^d$ the answer was known after a large series of papers over more than three decades.

As discussed, $r$-BP on the $d$-dimensional torus $\mathbb{T}_L^d$ goes through two phase transitions, which matches our intuitive argument from above. More precisely, the torus becomes fully white if $p\ll \mathscr{P}_2$, both color coexist if $\mathscr{P}_2\ll p\ll \mathscr{P}_1$, and it will become fully white if $\mathscr{P}_1\ll p$ a.s.s. where $\mathscr{P}_1=(\log_{(r-1)}L)^{-(d-r+1)}$ and $\mathscr{P}_2=L^{-d}$. The first transition has not been considered before, but it is very easy to handle. For $p\ll \mathscr{P}_2$, by a simple union bound the probability that there exists a black node in the initial configuration is upper-bounded by $L^d\cdot o(\mathscr{P}_2)=o(1)$, which implies a.a.s. the initial configuration is fully white. For $p\gg \mathscr{P}_2$, the expected number of black nodes in the initial configuration is equal to $L^d\cdot \omega(\mathscr{P}_2)=\omega(1)$; applying Chernoff bound~\cite{dubhashi2009concentration} yields that a.a.s. there exists a black node initially, which guarantees the survival of black color.

We prove that two-way $r$-BP on $\mathbb{T}_L^d$ exhibits a similar threshold behavior at threshold values $\mathscr{P}_1^{1/2^{r-1}}$ and $\mathscr{P}_2^{1/2^{r-1}}$. As mentioned, Balogh, Bollobas, Duminil-Copin, and Morris~\cite{balogh2012sharp} proved that $r$-BP actually exhibits a sharp threshold behavior in the second transition. Can we expect a sharp threshold in the first transition of $r$-BP or any of the two transitions of two-way $r$-BP? We believe that it might be the case in the second phase transition of two-way $r$-BP, but proving such a statement probably requires novel ideas beyond the known techniques in the literature. On the other hand, the claimed weak threshold behavior is the best possible for the first phase transition in both $r$-BP and two-way $r$-BP. (See the appendix, Section~\ref{no sharp}, for a simple proof of this claim.)   

\textbf{A more general statement.} Recall in Lemma~\ref{lemma 4} we proved that in two-way $r$-BP on $\mathbb{T}_L^d$ there is an $(r,b)$-eternal set of size $2^{r-1}$. In Lemma~\ref{lemma 1} we will show that actually there is no smaller $(r,b)$-eternal set. Therefore, by switching from $r$-BP to two-way $r$-BP, the minimum size of a $b$-eternal set increases from 1 to $2^{r-1}$. Thus, the threshold values in both $r$-BP and two-way $r$-BP are equal to $\mathscr{P}_1^{1/s}$ and $\mathscr{P}_2^{1/s}$, where $s$ is the minimum size of a $b$-eternal set. We believe that there exists a large class of monotone models $\mathscr{M}$ such that each model $M\in\mathscr{M}$ on $\mathbb{T}_L^d$ goes through two phase transitions at threshold values $\mathscr{P}_1^{1/s}$ and $\mathscr{P}_2^{1/s}$, where $s$ is the minimum size of a $b$-eternal set. We do not prove such a statement, but in Section~\ref{future work} we illustrate how our proof techniques can possibly be applied to provide such results. For now, let us discuss an interesting example which falls under the umbrella of the above argument. As an intermediate step from the analysis of $r$-BP to the analysis of two-way $r$-BP, Coker and Gunderson~\cite{coker2014sharp} studied the following variant of bootstrap percolation on $\mathbb{T}_L^2$, which is called \emph{2-BP with recovery}. In this model, a white node becomes black if it has at least two black neighbors, and a black node remains unchanged, except if all its four neighbors are white. Coker and Gunderson~\cite{coker2014sharp} proved that this process exhibits two phase transitions at $\mathscr{P}_1^{1/2}$ and $\mathscr{P}_2^{1/2}$. Notice this is consistent with the above claim because in $2$-BP with recovery the minimum size of a $b$-eternal set is 2. Clearly, one black node disappears in one round but two adjacent black nodes survive forever.

\textbf{Proof techniques.} Now, we discuss the high-level ideas of the proof techniques applied. In phase one, we want to show that if $p\ll \mathscr{P}_2^{1/2^{r-1}}$ then black color disappears a.a.s. We exploit a technique which we call \emph{clustering}; roughly speaking, we show $p$ is so small that a.a.s. one can partition all black nodes in small clusters which are far from each other. This distance lets us treat each cluster independently since there is no interaction among them. Furthermore, the number of black nodes in each cluster is less than $2^{r-1}$, that is the minimum size of an $(r,b)$-eternal set, which then results in the disappearance of black color. 

For the second phase, we must show both colors survive a.a.s. For black color, since there are $\Theta(L^d)$ pair-wise disjoint $(r,b)$-eternal sets of size $2^{r-1}$ (namely the even-part of $\Theta(L^d)$ pair-wise disjoint $r$-dimensional hyper-squares), applying Chernoff bound implies that there is a black $(r,b)$-eternal set initially a.a.s. This guarantees the survival of black color. We also need to show for $p\ll \mathscr{P}_1^{1/2^{r-1}}$ white color survives a.a.s. For that, we rely on the threshold behavior of $r$-BP. More precisely, applying the fact that the minimum size of a $b$-eternal set is equal to $2^{r-1}$, we show the probability that an arbitrary node is black after $T$ rounds, for some constant $T(d,r)$, is $o(\mathscr{P}_1)$. Since the stronger model of $r$-BP results in the survival of white color a.a.s. in this case, so does two-way $r$-BP. (Some details are removed.)

In phase 3, our goal is to prove if $\mathscr{P}_1^{1/2^{r-1}}\ll p$, the process a.a.s. becomes fully black. We utilize a method, which we call \emph{scaling}. The idea is to tile the torus $\mathbb{T}_L^d$ into $r$-dimensional hyper-squares and treat each of the hyper-squares as a single node. We say two hyper-squares are neighbors if there is at least one edge between them; then, each hyper-square has $2d$ neighbors, two in each dimension. Furthermore, we say a hyper-square is \emph{occupied} in configuration $\mathcal{C}_t$ for even $t$ (analogously odd $t$) if its even-part (resp. odd-part) is black. We prove if in some configuration, a hyper-square has occupied neighbors in $r$ distinct dimensions, then it becomes occupied in constantly many rounds. Furthermore, each hyper-square is occupied initially with probability $p^{2^{r-1}}\gg \mathscr{P}_1$. Hence, the process scaled to the hyper-squares is at least as strong as modified $r$-BP, where initially each hyper-square is occupied with probability $\omega(\mathscr{P}_1)$. We know modified $r$-BP with initial probability $\omega(\mathscr{P}_1)$ results in fully black configuration a.a.s. This implies that two-way $r$-BP on $\mathbb{T}_L^d$ reaches a configuration where the even-part of each of the hyper-squares is black a.a.s. We can do the same argument by switching the terms of odd and even in the definition of occupation. Then, by a union bound, a.a.s. the process becomes fully black. 

 
\textbf{Tie-breaking rule.} Let us finish this section, by mentioning an interesting observation. Another well-studied model in this literature is the \emph{majority model}, where by starting from an initial random configuration in each round all nodes update their color to the most frequent color in their neighborhood, and in case of a tie, a node keeps its current color. Two-way $r$-BP on $\mathbb{T}_L^d$ for $r=d$ is sometimes called the \emph{biased majority model} because each node selects the most frequent color in its neighborhood and in case of a tie, it chooses black (notice each node has degree $2d$). Therefore, the majority model on $\mathbb{T}_L^d$ is the same as the biased variant except in tie-breaking rule. We claim in the majority model if $p\leq 1-\delta$ for any arbitrary constant $\delta>0$, then the process does not become fully black a.a.s. For a simple proof see the appendix, Section~\ref{majority}. On the other hand as we discussed, in the biased model $p\gg \mathscr{P}_1^{1/2^{d-1}}$, consequently $p\geq \delta$ for an arbitrarily small constant $\delta>0$, results in fully black configuration a.a.s. Putting these two propositions in parallel, we observe that in the majority model $p$ should be very close to 1 to have a high chance of final complete occupancy by black, but by just changing the tie-breaking rule in favor of black, the process ends up in fully black configuration a.a.s. even for initial probability very close to 0. This comparison illustrates how small alternations in local behavior can result in considerable changes in the global behavior.
\vspace{-0.3 cm}
\section{Two Phase Transitions}
\label{two phase transition}
\subsection{Phase 1}
\label{phase 1}
The idea of the proof is to show that if $p\ll \mathscr{P}_2^{1/2^{r-1}}$, then a.a.s. black nodes in $\mathcal{C}_0$ are contained in a group of hyper-rectangles which are sufficiently far from each other and each hyper-rectangle includes less than $2^{r-1}$ black nodes. Since the hyper-rectangles are far from each other, the nodes out of the hyper-rectangles, which are all white initially, stay white forever (i.e., create a white $(r,w)$-robust set). Furthermore, the black nodes inside each hyper-rectangle die out after some rounds because they are less than the minimum size of an $(r,b)$-eternal set. We prove our claim in Theorem~\ref{theorem 1}, building on Lemma~\ref{lemma 1}. To prove Lemma~\ref{lemma 1}, we need to apply Lemma~\ref{lemma 3}, whose proof is given in the appendix, Section~\ref{lemma 3 proof}. 
\vspace{-0.2 cm}
\begin{lemma}
\label{lemma 3}
In $\mathbb{T}_L^d=(V,E)$, a non-empty $(r,b)$-robust set intersects at least $2^{r-1}$ pair-wise disjoint $(r,w)$-robust sets.
\end{lemma}
\vspace{-0.4 cm}
\begin{lemma}
\label{lemma 1}
In two-way $r$-BP on $\mathbb{T}_L^d$, a configuration with less than $2^{r-1}$ black nodes becomes fully white in $T$ rounds for some constant $T(d,r)$.
\end{lemma} 
\vspace{-0.4 cm}
\begin{proof}
Consider an initial configuration $\mathcal{C}_0$ which includes less than $2^{r-1}$ black nodes and denote the set of black nodes in $\mathcal{C}_0$ with $B$. Define the distance between two hyper-rectangles $HR$ and $HR'$ to be $d(HR,HR')=\min_{v\in HR, u\in HR'}d(v,u)$. Let $HR_1,\cdots, HR_k$ be constant-size hyper-rectangles whose pair-wise distance is at least three and include all black nodes. Notice such a set of hyper-rectangles exists since $|B|$ is a constant. All nodes which are not in the hyper-rectangles are white and remain white forever; that is, they are a white $(r,w)$-robust set. This is true because a node which is not in the hyper-rectangles is adjacent to at most one hyper-rectangle (otherwise it violates the aforementioned distance property), which implies at most one of its $2d$ neighbors is black. Therefore, only nodes in the hyper-rectangles can switch their color. Since the hyper-rectangles are of constant size, the number of configurations that the process can possibly reach from $\mathcal{C}_0$ is upper-bounded by some constant $T(d,r)$. That is, the process reaches a cycle of configurations in at most $T$ rounds. Based on the results by Goles and Olivos~\cite{goles1980periodic}, we know the length of the cycle is one or two.

Let $B'$ be the union of black nodes in the configuration(s) in the cycle; we claim $B'$ is an $(r,b)$-robust set. Therefore, if $B'$ is non-empty it must intersect at least $2^{r-1}$ pairwise disjoint $(r,w)$-robust sets based on Lemma~\ref{lemma 3}. However, initially there are at most $2^{r-1}-1$ black nodes, which can intersect at most $2^{r-1}-1$ pair-wise disjoint $(r,w)$-robust sets. Thus, there is at least one $(r,w)$-robust set which is initially fully white, but at the end includes a black node, which is a contradiction with its $(r,w)$-robustness. Thus, $B'$ is actually empty.

It only remains to show that $B'$ is $(r,b)$-robust. If the process reaches a cycle of length one, a fixed configuration, trivially the set of black nodes is an $(r,b)$-robust set. If it reaches a cycle of length two and switches between two configurations $\mathcal{C}_1$ and $\mathcal{C}_2$, we define $B'_1$ and $B'_2$ to be the set of black nodes in $\mathcal{C}_1$ and $\mathcal{C}_2$, respectively. The set $B'=B'_1\cup B'_2$ is $(r,b)$-robust since each node in $B'_1$ (similarly $B'_2$) has at least $r$ neighbors in $B'_2$ (resp. $B'_1$), otherwise it cannot be black in $\mathcal{C}_1$ (resp. $\mathcal{C}_2$). Therefore, each node in $B'$ has at least $r$ neighbors in $B'$, which implies it is an $(r,b)$-robust set. 
\end{proof}
\vspace{-0.3 cm}
\begin{theorem}
\label{theorem 1}
Two-way $r$-BP with $p\ll L^{-d/2^{r-1}}$ on $\mathbb{T}_L^d$ becomes fully white a.a.s.
\end{theorem}
\vspace{-0.4 cm}
\begin{proof}
Recall that the distance between two hyper-rectangles $HR$ and $HR'$ is equal to $d(HR,HR')=\min_{v\in HR, u\in HR'}d(v,u)$. We show for the initial configuration a.a.s. there is a set of hyper-rectangles which are pair-wise in distance at least three from each other and any black node belongs to one of these hyper-rectangles and the number of black nodes in each hyper-rectangle is less than $2^{r-1}$. Each node which is not in any of the hyper-rectangles is adjacent to at most one of them (otherwise, there are two hyper-rectangles whose distance is less than three). Thus, each of these nodes has at least $2d-1$ white neighbors which implies they all stay white forever. Furthermore, in each of these ``isolated'' hyper-rectangles there are less than $2^{r-1}$ black nodes which disappear after at most $T$ rounds by Lemma~\ref{lemma 1}.

It remains to prove that a.a.s. such a set of hyper-rectangles exist. For each black connected component in $\mathcal{C}_0$, consider the smallest hyper-rectangle which includes all its node. Let $\mathcal{A}_0$ be the set of these (not necessarily disjoint) hyper-rectangles. There is no black connected component of size $2^{r-1}$ or larger in $\mathcal{C}_0$ a.a.s. Let $X$ denote the number of black connected subgraphs of size $2^{r-1}$ in $\mathcal{C}_0$. The number of connected subgraphs of size $2^{r-1}$ which include an arbitrary node $v$ is a constant (notice $d$, thus also $r$, is fixed); then, the number of connected subgraphs of size $2^{r-1}$ is of order $\Theta(L^d)$. Thus, $\mathbb{E}[X]=\Theta(L^d)\ p^{2^{r-1}}=\Theta(L^d)\ o(L^{-d})=o(1)$. By Markov's inequality~\cite{dubhashi2009concentration} a.a.s. there is no black connected subgraph of size $2^{r-1}$, which implies there is no black connected component of this size or larger. Therefore, for any hyper-rectangle of size $l_1\times\cdots \times l_d$ in $\mathcal{A}_0$, $l_j<2^{r-1}$ for all $1\leq j\leq d$ a.a.s.

Consider the following procedure. By starting from $\mathcal{A}=\mathcal{A}_0$, in each iteration if all hyper-rectangles in $\mathcal{A}$ are pair-wise in distance at least three from each other, the procedure is over, otherwise there are two hyper-rectangles $HR_1,HR_2\in \mathcal{A}$ such that $d(HR_1,HR_2)\leq 2$. In this case, we set $\mathcal{A}=\mathcal{A}\setminus\{HR_1,HR_2\}\cup\{HR\}$, where $HR$ is the smallest hyper-rectangle which includes all black nodes in both $HR_1$ and $HR_2$. See Figure~\ref{fig 3} (a) and (b) for an example, where the boundaries of the smallest hyper-rectangles are distinguished by green. The process definitely terminates, because in each round $|\mathcal{A}|$ decreases. Moreover, when the process is over, the hyper-rectangles in $\mathcal{A}$ satisfy our desired distance property. We still have to show that each of them contains less than $2^{r-1}$ black nodes. Let us make the three following observations.
\begin{figure}[h]
\begin{center}
\includegraphics[width=0.7	\textwidth]{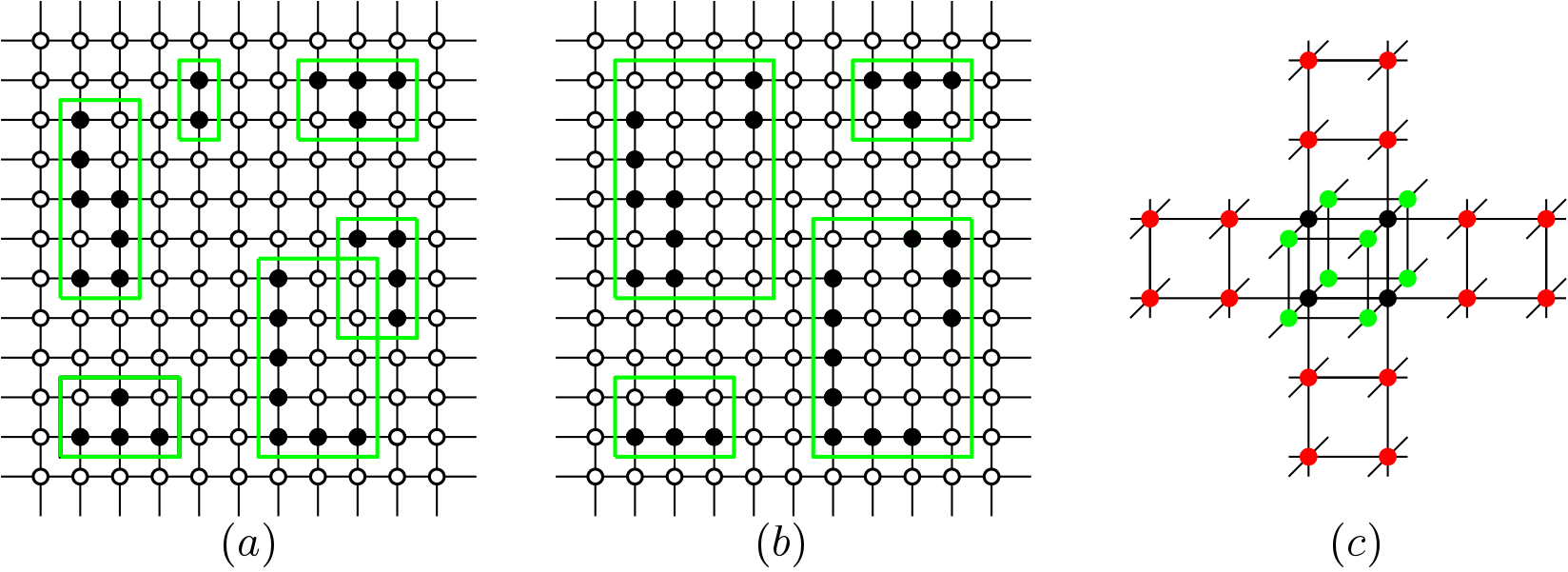}
\caption{(a) the smallest hyper-rectangles (b) after two iterations (c) the inner and outer neighbors.\label{fig 3}}
\end{center}
\vspace{-0.5 cm}
\end{figure}

(a) Let $HR$ of size $l_1\times\cdots\times l_d$ be the smallest hyper-rectangle which contains all black nodes in both $HR_1$ and $HR_2$ respectively of size $l^{(1)}_1\times\cdots\times l^{(1)}_d$ and $l^{(2)}_1\times\cdots\times l^{(2)}_d$ in the above procedure, we have $l_j\leq 3\max_{i\in\{1,2\}}l^{(i)}_j$ for all $1\leq j\leq d$ because $d(HR_1,HR_2)\leq 2$.

(b) Assume a hyper-rectangle $HR$ of size $l_1\times\cdots\times l_d$ starting in $(i_1,\cdots,i_d)$ is in $\mathcal{A}$ at some iteration in the above procedure, then it contains at least $\max_{1\leq j\leq d} l_j/2$ black nodes. Intuitively, this should be obvious since in each iteration we combine two hyper-rectangles whose distance is at most two. For a formal proof, let us first show that for any two black nodes $v,u$ in a hyper-rectangle $HR$ in $\mathcal{A}$, there is a semi-connected path between $v$ and $u$ along the black nodes in $HR$. We apply proof by induction; initially, this is trivially true since each hyper-rectangle includes a black connected component. Assume in $k$-th iteration we combine $HR_1$ and $HR_2$ because there is node $v'$ in $HR_1$ and node $u'$ in $HR_2$ such that $d(v',u')\leq 2$. In the new hyper-rectangle $HR$, every two black nodes originally from $HR_1$ (similarly from $HR_2$) are semi-connected by the induction hypothesis. Two black nodes $v$ and $u$ respectively from $HR_1$ and $HR_2$ are also semi-connected along a semi-connected path from $v$ to $v'$, from $v'$ to $u'$, and finally from $u'$ to $u$. Assume $l_{j'}=\max_{1\leq j\leq d}l_j$; since $HR$ is the smallest hyper-rectangle, there is a black node whose $j'$-th coordinate is $i_{j'}$ and a black node whose $j'$-th coordinate is $i_{j'}+l_{j'}$. Consider the semi-connected path between these two nodes which clearly includes at least $l_{j'}/2$ black nodes.

(c) In $\mathcal{C}_0$ a.a.s. there is no hyper-rectangle $HR$ of size $l_1\times\cdots\times l_d$ which includes at least $2^{r-1}$ black nodes and $l_j< 6\cdot 2^{r-1}$ for all $1\leq j\leq d$. Let random variable $Y$ denote the number of such hyper-rectangles. The number of hyper-rectangles of the aforementioned sizes starting from a fixed node $i$ is bounded by constant $K=(6\cdot2^{r-1})^{d}$, which implies there are at most $KL^d$ hyper-rectangles of such sizes. Thus, $\mathbb{E}[Y]\leq KL^d {K \choose 2^{r-1}}p^{2^{r-1}}=o(1)$ for $p=o(L^{-\frac{d}{2^{r-1}}})$, which implies $Y=0$ a.a.s. by applying Markov's inequality.

At the beginning of the proof we showed that all the sides of any hyper-rectangle in $\mathcal{A}_0$ are smaller than $2^{r-1}$ a.a.s. Putting this fact in parallel with (a), we conclude if the process does not terminate while all the sides of any hyper-rectangle in $\mathcal{A}$ are smaller than or equal to $2\cdot 2^{r-1}$, then it has to generate a hyper-rectangle $HR'$ of size $l'_1\times\cdots\times l'_d$ such that $\forall\ 1\leq j\leq d$, $l'_j\leq 6\cdot2^{r-1}$ and there exists $1\leq j'\leq d$ such that $2\cdot 2^{r-1}<l'_{j'}$. Based on (b), $HR'$ must include at least $l'_{j'}/2 \geq 2^{r-1}$ black nodes; however, based on (c) such an $HR'$ does not exist a.a.s. Therefore, a.a.s. the process terminates while all the sides of any hyper-rectangle in $\mathcal{A}$ are upper-bounded by $2\cdot 2^{r-1}$. By applying (c) another time, none of the hyper-rectangles includes $2^{r-1}$ or more black nodes a.a.s. 
\end{proof}
\vspace{-0.6 cm}
\subsection{Phase 2}
\label{phase 2}
In this section, we prove that two-way $r$-BP with $L^{-\frac{d}{2^{r-1}}}\ll p\ll (\log_{(r-1)}L)^{-\frac{d-r+1}{2^{r-1}}}$ on $\mathbb{T}_L^d=(V,E)$ results in the stable coexistence of both colors a.a.s.

Let us first show that black color a.a.s. will survive for $L^{-\frac{d}{2^{r-1}}}\ll p$. As discussed, in $\mathbb{T}_L^d$ there are $L^d/\gamma$ pair-wise disjoint $r$-dimensional hyper-squares, for a constant $\gamma\simeq 2^r$. Consider an arbitrary labeling from 1 to $L^d/\gamma$ on these hyper-squares and define Bernoulli random variable $x_k$ for $1\leq k\leq L^d/\gamma$ to be 1 if the even-part of $k$-th hyper-square is fully black in $\mathcal{C}_0$ and let $X:=\sum_{k=1}^{L^d/\gamma}x_k$. We show $X\ne 0$ a.a.s., which implies that there is a hyper-square whose even-part is fully black initially. Since the even-part of an $r$-dimensional hyper-square is an $(r,b)$-eternal set (see Lemma~\ref{lemma 4}), it guarantees the survival of black color. We have $\mathbb{E}[X]=(L^d/\gamma) p^{2^{r-1}}=(L^d/\gamma) \omega(1/L^d)=\omega(1)$, where we used that the even-part of an $r$-dimensional hyper-square is of size $2^{r-1}$. Since $X$ is the sum of independent Bernoulli random variables, we have $Pr[X=0]\leq \exp(-\omega(1))=o(1)$ by Chernoff bound.

It remains to prove that white color survives a.a.s. if $p\ll \mathscr{P}_1^{1/2^{r-1}}$. Based on Lemma~\ref{lemma 1}, there is a constant $T$ so that by starting from an initial configuration with less than $2^{r-1}$ black nodes, we have no black nodes after $T$ rounds. We claim this implies that for an arbitrary node $v$ to be black in round $T$, it needs at least $2^{r-1}$ black nodes in its $T$-neighborhood (i.e., nodes in distance at most $T$ from $v$) in the initial configuration. For the sake of contradiction, assume that there is an initial configuration $\mathcal{C}_0$ in which $v$ has less than $2^{r-1}$ black nodes in its $T$-neighborhood and it is black in the $T$-th round. Then, we consider the initial configuration $\mathcal{C}'_0$ in which all nodes in $v$'s $T$-neighborhood have the same color as $\mathcal{C}_0$ and all others are white. Configuration $\mathcal{C}'_0$ has less than $2^{r-1}$ black nodes. Furthermore, $v$ must be black after $T$ rounds by starting from $\mathcal{C}'_0$ because the color of $v$ in round $T$ is only a function of the initial color of nodes in its $T$-neighborhood (this is easy to see; however, for a formal proof one can simply apply induction) and the color of all nodes in the $T$-neighborhood of $v$ is the same as $\mathcal{C}_0$. However, this is in contradiction with Lemma~\ref{lemma 1}.

So far, we know that for a node $v$ to be black in round $T$, it needs at least $2^{r-1}$ black nodes in its $T$-neighborhood initially. This immediately implies that for an arbitrary node, the probability of being black in round $T$ is upper-bounded by $Kp^{2^{r-1}}=o(\mathscr{P}_1)$, where constant $K$ is an upper-bound on the number of possibilities of choosing $2^{r-1}$ nodes in the $T$-neighborhood of an arbitrary node in $\mathbb{T}_L^d$; notice since $d$ and $T$ both are constant, the number of nodes in $T$-neighborhood of a node is bounded by a constant. Therefore, in round $T$ each node is black with probability $o(\mathscr{P}_1)$. It is known (see Theorem~\ref{theorem 2}) that the stronger model of $r$-BP results in the survival of white color from such a configuration a.a.s., so does two-way $r$-BP. The first part of the last statement is not fully correct since in $r$-BP each node is black independently, but here clearly the color of a node is not independent from the color of nodes in its $2T$-neighborhood. In the appendix, Section~\ref{dependency}, we show that the proof of Theorem~\ref{theorem 2} is robust enough to tolerate this level of local dependency.
\vspace{-0.2 cm}
\begin{theorem}
\cite{cerf2002threshold}
\label{theorem 2}
In $r$-BP on $\mathbb{T}_L^d$ if $p\ll \mathscr{P}_1$, white color will survive forever a.a.s. 
\end{theorem}
\vspace{-0.5 cm}
\subsection{Phase 3}
\label{phase 3}
In this section, we prove that in two-way $r$-BP on $\mathbb{T}_L^d$, if $\mathscr{P}_1^{1/2^{r-1}}\ll p$ then a.a.s. the process becomes fully black, where $\mathscr{P}_1=(\log_{(r-1)}L)^{-(d-r+1)}$. For the sake of simplicity, assume $L$ is even (we discus at the end, how our argument easily carries on the odd case). Recall that the tiling procedure (from Section~\ref{definitions}) partitions the node set of $\mathbb{T}_L^d$ into $L^d/2^r$ pair-wise disjoint $r$-dimensional hyper-squares. We say two hyper-squares are neighbors if their distance is equal to one, i.e., there is an edge between them. More precisely, the neighbors of an $r$-dimensional hyper-square $HS$ starting from $i=(i_1,\cdots,i_d)$ are divided into two groups. First, $2r$ hyper-squares whose starting nodes differ with $i$ only in one of the first $r$ coordinates and exactly by two, which are called the \emph{inner} neighbors. The second group are $2(d-r)$ hyper-squares whose starting nodes differ with $i$ in only one of the last $d-r$ coordinates and exactly by one, which are called the \emph{outer} neighbors. The two hyper-squares whose starting nodes differ with $i$ in the $j$-th coordinate are called the neighbors in the $j$-th dimension. See Figure~\ref{fig 3} (c) for an example of the inner (red) and outer (green) neighbors of a 2-dimensional hyper-square in $\mathbb{T}_L^3$. Furthermore, let us define the \emph{parity} of $HS$ to be the parity of the sum of the last $d-r$ coordinates of $i$. Clearly, the inner neighbors have the same parity as $HS$ but the outer neighbors have different parity. 

From now on, we only look at the even rounds; i.e., we only consider $\mathcal{C}_t$ for even $t$. For an $r$-dimensional hyper-square of even parity (similarly odd parity), we say it is \emph{occupied} in $\mathcal{C}_t$ if its even-part (resp. odd-part) is black. Based on Lemma~\ref{lemma 4}, an occupied hyper-square remains occupied forever. In Lemma~\ref{lemma 5}, we state that if in some configuration in two-way $r$-BP on $\mathbb{T}_L^d$, an $r$-dimensional hyper-square has occupied neighbor in at least $r$ distinct dimensions then it becomes occupied in constantly many rounds. The proof is technical and is presented in the appendix, Section~\ref{proof of lemma 5}. The idea is to apply induction on $r$. See Figure~\ref{fig 4}, for two examples on how a 2-dimensional hyper-square becomes occupied with occupied neighbors in two distinct dimensions (regarding the selection of black nodes, recall that the parity of a hyper-square is the same as its inner neighbors but different with outer ones).
\vspace{-0.5 cm}
\begin{figure}[h]
\begin{center}
\includegraphics[width=0.6	\textwidth]{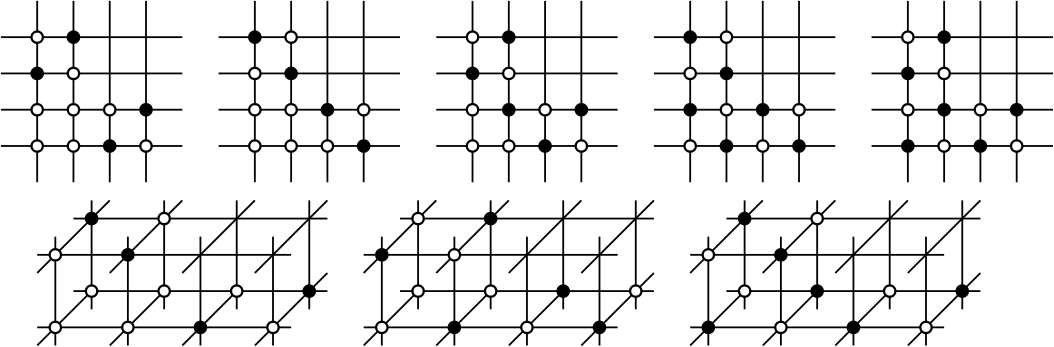}
\caption{(top) two inner occupied neighbors (bottom) one inner and one outer occupied neighbor. \label{fig 4}}
\end{center}
\end{figure}
\vspace{-0.8 cm}
\begin{lemma}
\label{lemma 5}
In two-way $r$-BP on $\mathbb{T}_L^d$, if an $r$-dimensional hyper-square has occupied neighbor in at least $r$ distinct dimensions, it becomes occupied in $t'$ rounds for some even constant $t'$.
\end{lemma} 
For our proof, we also need that modified $r$-BP on $\mathbb{T}_L^d$ with initial probability $\omega(\mathscr{P}_1)$ results in fully black configuration a.a.s. However, this is known only for $d=r$ by Holroyd~\cite{holroyd2006metastability}. He showed that the process exhibits a sharp threshold behavior at $\lambda' \mathscr{P}_1$ for some constant $\lambda'>0$. We require a much weaker statement; that is, the initial probability $\omega(\mathscr{P}_1)$ a.a.s. results in fully black configuration, but for all values of $r\le d$. The good news is that the upper bound proof by Cerf and Manzo~\cite{cerf2002threshold} regarding $r$-BP can be easily adapted to prove our desired upper bound for modified $r$-BP. Actually, exactly the same proof works because wherever they apply $r$-BP rule, modified $r$-BP suffices. However, it is interesting by its own sake to study the sharp threshold behavior of modified $r$-BP also for $r\ne d$, in future work.
\begin{theorem} (derived from~\cite{cerf2002threshold})
\label{theorem 3}
In modified $r$-BP on $\mathbb{T}_L^d$ for $r\le d$, if $p=\omega(\mathscr{P}_1)$, then the process becomes fully black. 
\end{theorem}   
Now, it is time to put the aforementioned claims together to finish the proof. If we tile $\mathbb{T}_L^d$ into hyper-squares as above, in two-way $r$-BP with $p=\omega(\mathscr{P}_1^{1/2^{r-1}})$ each hyper-square is occupied initially with probability $\omega(\mathscr{P}_1)$. Furthermore, based on Lemma~\ref{lemma 5} if a hyper-square has occupied neighbor in at least $r$ distinct dimensions, it becomes occupied, which implies the occupation process among the hyper-squares is at least as strong as modified $r$-BP. Based on Theorem~\ref{theorem 3}, we know modified $r$-BP with initial probability $\omega(\mathscr{P}_1)$ becomes fully black a.a.s. Thus, all the hyper-squares become occupied in our process a.a.s. We can do the same argument by just switching the terms of even and odd in the definition of occupation. Then, by a union bound, a.a.s. for two-way $r$-BP on $\mathbb{T}_L^d$ with $p=\omega(\mathscr{P}_1^{1/2^{r-1}})$ eventually both the even-part and odd-part of all the hyper-squares are black, which implies that the process becomes fully black. 

We assumed at the beginning that $L$ is even. Theorem~\ref{theorem 3} also works for the $d$-dimensional lattice $[L]^{d}$. Therefore, for odd $L$ we can do the same argument for the lattice, attained by skipping the nodes with at least one coordinate equal to $L$, and also the lattice, attained by skipping the nodes with at least one coordinate equal to one. Then, again a union bound finishes the proof.
\vspace{-0.3cm}
\section{Future Work}
\label{future work}
We proved that two-way $r$-BP on $\mathbb{T}_L^d$ exhibits a threshold behavior with two phase transitions at $\mathscr{P}_1^{1/s}$ and $\mathscr{P}_2^{1/s}$ where $s$ is the minimum size of a $b$-eternal set. The question, then, arises: Can one prove such results for a larger class of models? We introduce a sub-class of monotone models on the $d$-dimensional torus and then explain how one can possibly employ our proof techniques to prove the desired threshold behavior in this more general framework. 

In $(r,r')$-BP on $\mathbb{T}_L^d$ and for $0\le r'\le r\le d$, by starting from an initial random configuration in discrete-time rounds each white node becomes black if and only if it has at least $r$ black neighbors and each black node remains black if and only if it has at least $r'$ black neighbors. This includes $\sum_{r=0}^{d}\sum_{r'=0}^{r'=r}1=(d+1)(d+2)/2$ different models on $\mathbb{T}_L^d$. For instance, $(r,0)$-BP, $(r,r)$-BP, and $(r,1)$-BP are respectively the same as $r$-BP, two-way $r$-BP, and $r$-BP with recovery. It is an interesting exercise to check that $(r,r')$-BP for $0\le r'\le r$ includes all monotone models. We add the constraint $r\le d$ to make sure that the model includes a constant-size $b$-eternal set (and no constant-size $w$-eternal set). Note that monotonicity of the model and constant-size $b$-eternal set are inseparable parts of our proof techniques.

Now, we illustrate by applying our proof techniques, some prior results, and some novel ideas one can possibly prove that $(r,r')$-BP on $\mathbb{T}_L^d$ goes through two phase transitions at $\mathscr{P}_1^{1/s}$ and $\mathscr{P}_2^{1/s}$, for $s$ being the minimum size of a $b$-eternal set. Notice that for $0\le r'\le r\le d$, an $r'$-dimensional hyper-square is a $b$-eternal set, which implies that $s$ is a constant smaller than $2^{r'}$. We assume that $r\ge 2$.
\begin{itemize}
\item Phase 1: We can show that the process becomes fully white if $p\ll \mathscr{P}_2^{1/s}=L^{-d/s}$, by replacing $2^{r-1}$ with $s$ in the proof of Theorem~\ref{theorem 1}, where the clustering technique is applied.
\item Phase 2: There are $\Theta(L^d)$ pair-wise disjoint $r'$-dimensional hyper-squares and each of them includes a $b$-eternal set of size $s$. For $L^{-d/s}=\mathscr{P}_2^{1/s}\ll p$, in expectation $\Theta(L^d)p^s=\omega(1)$ of these $b$-eternal sets are fully black in the initial configuration. Therefore, by applying Chernoff bound a.a.s. there is a fully black $b$-eternal set in the initial configuration. For $p\ll \mathscr{P}_1^{1/s}$, employing our argument from Section~\ref{phase 2} implies that after a constant number of rounds, each node is black with probability $\Theta(p^s)=o(\mathscr{P}_1)$. We know that the stronger model of $r$-BP results in the survival of white color from such a configuration a.a.s., so does $(r,r')$-BP. (Again, we clearly have the dependency issue, which can be handled by the argument from Section~\ref{dependency}.) 

\item Phase 3: We can apply the scaling technique by tiling the torus into $r'$-dimensional hyper-squares. However, to ``reduce'' this scaled process to modified $r$-BP, we need some knowledge about the structure of the $b$-eternal sets in addition to the value of $s$. We believe that one can extract sufficient structural properties such as symmetry from the definition of $(r,r')$-BP, but this is left for future work. 
\end{itemize}

In the present paper, we studied the random setting, but from an extremal point of view it is natural to ask: What is the minimum number of nodes which must be black initially to make the whole graph black? This question has been studied extensively for both $r$-BP and two-way $r$-BP on $\mathbb{T}_L^d$, see e.g.~\cite{balogh1998random,flocchini2004dynamic,balister2010random,morrison2018extremal,hambardzumyan2017polynomial,jeger2019dynamic}, and some lower and upper bounds are known. Can our proof techniques be used to improve on these bounds?

It is also interesting to study the \emph{expected consensus time} of the process, which is the expected number of rounds the process needs to reach a cycle of configurations for an initial random configuration. We are not aware of any result for two-way $r$-BP on $\mathbb{T}_L^d$, and for $r$-BP, the answer is known only for $d=2$, by Balister, Bollobas, and Smith~\cite{balister2016time}. 


\textbf{Acknowledgments.} The author likes to thank Raphael Cerf, Bernd G\"artner, and Roberto H. Schonmann for several stimulating discussions.
\bibliography{ref}
\newpage
\appendix
\section{Appendix}
\subsection{No Sharp Threshold in First Transition}
\label{no sharp}
In $r$-BP on $\mathbb{T}_L^d$, the process exhibits a sharp threshold behavior in the second phase transition; that is, if $p\ge (1+\epsilon)\lambda \mathscr{P}_1$ (analogously $p\le (1-\epsilon)\lambda \mathscr{P}_1$) for some fixed constant $\lambda(d,r)>0$ the process becomes (resp. does not become) fully black a.a.s. for any constant $\epsilon>0$. We claim that one cannot expect such a behavior in the first transition. We show that if $p=\mu\mathscr{P}_2$ for any constant $\mu$, then black color survives with some non-zero constant probability (which implies that there is no constant $\mu'$ such that $p\le (1-\epsilon)\mu' \mathscr{P}_2$ results in fully white configuration a.a.s. for any constant $\epsilon>0$). The probability that all nodes are white initially is equal to $(1-p)^{L^d}$ which is smaller than $\exp(-pL^d)=\exp(-\mu)$ for $p=\mu\mathscr{P}_2$, where we used the estimate $1-x\leq \exp(-x)$. Thus, there is a black node initially with a non-zero constant probability. 

Now by applying a similar argument, we show that one cannot also expect a sharp threshold behavior in the first phase transition of two-way $r$-BP on $\mathbb{T}_L^d$. We prove that if $p=\mu \mathscr{P}_2^{1/2^{r-1}}$ for any constant $\mu>0$, then black color has a constant non-zero probability to survive. There are $L^d/\gamma$ pair-wise-disjoint $r$-dimensional hyper-squares in $\mathbb{T}_L^d$ for some constant $\gamma>0$ and based on Lemma~\ref{lemma 4} the even-part of an $r$-dimensional hyper-square is an $(r,b)$-eternal set. The probability that the even-part of at least one of these disjoint hyper-square is black initially is equal to $1-(1-p^{2^{r-1}})^{L^d/\gamma}$. Again by applying the estimate $1-x\leq \exp(-x)$, this probability is lower-bounded by $1-\exp(-\frac{\mu^{2^{r-1}}}{\gamma})$, which is a non-zero constant. 

\subsection{Majority Model}
\label{majority}
 \begin{theorem}
 In the majority model on $\mathbb{T}_L^d$ if $p\le 1-\delta$ for any constant $\delta>0$, then white color survives a.a.s.
 \end{theorem}
\begin{proof}
There are $L^d/\gamma$ pair-wise disjoint $d$-dimensional hyper-squares in $\mathbb{T}_L^d$ for some constant $\gamma\simeq 2^d$ as discussed at the end of Section~\ref{definitions}. Let us label them from $1$ to $L^d/\gamma$ and define Bernoulli random variable $x_k$ to be one if the $k$-th hyper-square is fully white in the initial configuration for $1\leq k\leq L^d/\gamma$. Let $X:=\sum_{k=1}^{L^d/\gamma}x_k$. Since each hyper-square has $2^d$ nodes, $\mathbb{E}[X]=\frac{L^d}{\gamma}(1-p)^{2^d}=\Omega(L^d)$ for $p\le 1-\delta$. Since $x_k$s are independent, by applying Chernoff bound a.a.s. there is a fully white $d$-dimensional hyper-square initially, which is a $w$-eternal set in the majority model. This is true because as we argued in Lemma~\ref{lemma 4} each node in a $d$-dimensional hyper-square $HS$ has exactly $d$ neighbors in $HS$.
\end{proof}
\subsection{Proof of Lemma~\ref{lemma 3}}
\label{lemma 3 proof}
\textbf{Lemma~\ref{lemma 3}.} \textit{In $\mathbb{T}_L^d=(V,E)$, a non-empty $(r,b)$-robust set intersects at least $2^{r-1}$ pair-wise disjoint $(r,w)$-robust sets.}
\begin{proof}
We do induction on $r$. As the base case we show that a $(2,b)$-robust set $S$ in $\mathbb{T}_L^d$ intersects at least $2^{2-1}=2$ disjoint $(2,w)$-robust sets $W_1$ and $W_2$. There exists some coordinate $j$ so that there are two nodes $i^{(1)}=(i_{1}^{(1)},\cdots,i_{d}^{(2)})$ and $i^{(2)}=(i_{1}^{(2)},\cdots,i_{d}^{(2)})$ in $S$ with $i_{j}^{(1)}< i_{j}^{(2)}$ (otherwise $S$ includes only one node, which cannot be a $(2,b)$-robust set). Let $W_1=\{(i_1,\cdots,i_d)\in V:i_j=i_j^{(1)}\vee i_j=i_j^{(1)}-1\}$ and $W_2=V\setminus W_1$. Notice that $W_1$ includes $i^{(1)}$ and $W_2$ includes $i^{(2)}$, except if $i_{j}^{(1)}=1$ and $i_{j}^{(2)}=L$, but in this case we simply use $i_{j}^{(1)}+1$ instead of $i_{j}^{(1)}-1$ in $W_1$. For any node $i\in W_1$ (similarly in $W_2$), $2d-2$ of neighbors which differ with $i$ only in some coordinate $j'\ne j$ are all in $W_1$ (resp. $W_2$) and among the two neighbors which differ in the $j$-th coordinate at least one of them is in $W_1$ (resp. $W_2$) by construction. Thus, each node has at least $2d-1$ of its $2d$ neighbors in $W_1$ (resp. $W_2$), which implies it is a $(2,w)$-robust set.

Now, as the induction hypothesis assume that the statement is true for some $r\geq 2$, we show it holds also for $r+1$. Let set $S$ be an $(r+1,b)$-robust set. There exists some coordinate $j$ so that there are two nodes in $S$ which differ in the $j$-th coordinate, otherwise it includes only one node. Let level $L_k$ be the nodes whose $j$-th coordinate is $k$ for $1\leq k\leq L$. In other words, level $L_k$ is the node set of the $(d-1)$-dimensional torus attained by fixing the $j$-th coordinate to be $k$. Based on above, we know there are at least two levels which intersect $S$. Assume there are $1 \leq k_1,k_2\leq L$ such that $L_{k_1}$ and $L_{k_2}$ intersect $S$ but $L_{k_1+1}$ and $L_{k_2-1}$ do not and $(L_{k_1}\cup L_{k_1+1})\cap (L_{k_2-1}\cup L_{k_2})=\emptyset$. In other words, there are two disjoint pairs and each pair includes two adjacent levels, where one level intersects $S$ and the other one does not. If such pairs do not exist then there are $\Theta(L)$ levels which intersect $S$. Furthermore, each set $L_{2k-1}\cup L_{2k}$ for $1 \leq k \leq \lfloor L/2 \rfloor$ is an $(r+1,w)$-robust set because each node in $L_{2k-1}\cup L_{2k}$ has exactly $2d-1$ neighbors in it. Based on the last two statements if there do not exist such disjoint pairs of levels, we have $\Theta(L) \geq 2^{(r+1)-1}=2^r$ pair-wise disjoint $(r+1,w)$-robust sets which intersect $S$, which then we are done. Therefore, assume such disjoint pairs of levels $L_{k_1}\cup L_{k_1+1}$ and $L_{k_2-1}\cup L_{k_2}$ exist.

We define $S_1:=S\cap L_{k_1}$ and $S_2:=S\cap L_{k_2}$; let $(d-1)$-dimensional torus $\mathbb{T}_1$ (similarly $\mathbb{T}_2$) be the induced subgraph on node set $L_{k_1}$ (resp. $L_{k_2}$). We claim each node in $S_1$ (similarly $S_2$) has at least $r$ neighbors in $S_1$ (resp. $S_2$), which means $S_1$ (resp. $S_2$) is an $(r,b)$-robust set with respect to $\mathbb{T}_1$ (resp. $\mathbb{T}_2$). We prove the claim for $S_1$, and the proof for $S_2$ is analogous. Each node in $L_{k_1}$ has all its neighbors in $L_{k_1}$ except one in $L_{k_1-1}$ and one in $L_{k_1+1}$, but the one in $L_{k_1+1}$ is not in $S$ because $L_{k_1+1}\cap S=\emptyset$ by our assumption. Furthermore, each node in $S$ has at least $r+1$ neighbors in $S$, which implies that each node in $S_1$ has at least $r$ neighbors in $S_1$. Since $S_1$ is an $(r,b)$-robust set with respect to the $(d-1)$-dimensional torus $\mathbb{T}_1$, it intersects at least $2^{r-1}$ pair-wise disjoint $(r,w)$-robust sets in $\mathbb{T}_1$ by the induction hypothesis. The same argument applies to $S_2$ with respect to $\mathbb{T}_2$. Therefore, there are pair-wise disjoint sets $W_1^{(1)},\cdots,W_{2^{r-1}}^{(1)}\subset L_{k_1}$ (analogously $W_{2^{r-1}+1}^{(1)},\cdots,W_{2^{r}}^{(1)}\subset L_{k_2}$) such that a node $v$ in $W_{\ell}^{(1)}$ for $1\leq \ell \leq 2^{r-1}$ (resp. $2^{r-1}+1\leq \ell \leq 2^{r}$) has at least $2(d-1)-r+1$ neighbors in $W_{\ell}^{(1)}$, based on the definition of an $(r,w)$-robust set in a $(d-1)$-dimensional torus. Now, let set $W_{\ell}^{(2)}$ for $1\leq \ell \leq 2^{r-1}$ (similarly $1\leq \ell \leq 2^{r-1}$) be the mapping of set $W_{\ell}^{(1)}$ into $L_{k_1+1}$ (resp. $L_{k_2-1}$); that is, we change the $j$-th coordinate from $k_1$ (resp. $k_2$) to $k_1+1$ (resp. $k_2-1$) for each node in $W_{\ell}^{(1)}$ to obtain $W_{\ell}^{(2)}$. Now, we claim $W_{\ell}:=W_{\ell}^{(1)}\cup W_{\ell}^{(2)}$ for $1\leq \ell \leq 2^{r}$ are $2^r$ pair-wise disjoint $(r+1,w)$-robust sets in $\mathbb{T}_L^d$ which all intersect $S$. Firstly, each node in $W_{\ell}^{(1)}$ (similarly $W_{\ell}^{(2)}$) has at least $2(d-1)-r+1=2d-r-1$ neighbors in $W_{\ell}^{(1)}$ (resp. $W_{\ell}^{(2)}$) and one neighbor in $W_{\ell}^{(2)}$ (resp. $W_{\ell}^{(1)}$), which is overall $2d-r=2d-(r+1)+1$ neighbors in $W_{\ell}$; this implies that it is an $(r+1,w)$-robust set in $\mathbb{T}_L^d$. Furthermore, they are all disjoint because based on our construction $(L_{k_1}\cup L_{k_1+1})\cap (L_{k_2-1}\cup L_{k_2})=\emptyset$. Finally, $S$ intersects each $W_{\ell}$ for $1\leq \ell \leq 2^{r}$ because it intersects $W_{\ell}^{(1)}$ based on the induction hypothesis.  
\end{proof}

\subsection{Locally Dependent $\textbf{r}$-BP}
\label{dependency}
We want to prove that in two-way $r$-BP on $\mathbb{T}_L^d$ if $p \ll \mathscr{P}_1^{1/2^{r-1}}$ then a.a.s. white color survives forever. In Section~\ref{phase 2}, using Lemma~\ref{lemma 1} we showed that there is a constant $T(d,r)$ such that for each node to be black in the $T$-th round, it needs to have at least $2^{r-1}$ black nodes in its $T$-neighborhood initially. Let us introduce a new process on $\mathbb{T}_L^d$. Assume that the initial configuration is obtained in the following way: first we make each node black independently with probability $p^{1/2^{r-1}}$, and then each node will be assigned black color if it has at least $2^{r-1}$ black nodes in its $T$-neighborhood and white otherwise. Starting from such initial configuration, in each round a white node becomes black if it has at least $r$ black neighbors and black nodes remain unchanged. We call this process \emph{locally dependent $r$-BP}. Clearly, if we prove that in locally dependent $r$-BP on $\mathbb{T}_L^d$ for $p\ll \mathscr{P}_1$, a.a.s. white color survives, then we are done due to the monotonicity of two-way $r$-BP. To prove this statement, we rely on the results by Cerf and Manzo~\cite{cerf2002threshold} who showed that in $r$-BP on $\mathbb{T}_L^d$, white color will survive forever a.a.s. if $p\ll\mathscr{P}_1$. We show that a careful treatment of their proof results in the same statement for locally dependent $r$-BP. Note that the setting of $r$-BP with $p\ll \mathscr{P}_1$ is the same as locally dependent $r$-BP with $p\ll \mathscr{P}_1$. Firstly, they follow the same updating rule. Secondly, each node is black initially with probability $o(\mathscr{P}_1)$. This is trivial in $r$-BP and it is true in locally dependent $r$-BP since the number of possibilities of choosing $2^{r-1}$ nodes in the $T$-neighborhood of an arbitrary node in $\mathbb{T}_L^d$ is bounded by a constant, where we use that $r$, $d$, and $T$ are constants. The only difference is that in $r$-BP each node is black independently from all other nodes, but in locally dependent $r$-BP each node is black independently from all nodes which are not in its $2T$-neighborhood. (Two nodes which are in distance $2T$ or smaller are not independent since their $T$-neighborhood overlaps.) 
\begin{theorem} (derived from~\cite{cerf2002threshold})
\label{theorem 4}
In locally dependent $r$-BP on $\mathbb{T}_L^d$, white color will survive forever a.a.s. if $p=o(\mathscr{P}_1)$.
\end{theorem}

Cerf and Manzo~\cite{cerf2002threshold} proved the statement of Theorem~\ref{theorem 4} for $r$-BP. However, basically the same proof with some small changes can be applied to prove Theorem~\ref{theorem 4}. The main ingredients of their proof are two lemmata, Lemma 5.1 and Lemma 5.2 in their paper. The proof of Lemma 5.2 (originally proved by Aizenman and Lebowitz~\cite{aizenman1988metastability}) and how to combine these two lemmata to prove the final statement is quite straightforward and does not use the independence in the initial configuration. Thus it remains to show that the statement of Lemma 5.2 is also true for locally dependent $r$-BP, which we present in Lemma~\ref{lemma 6}.

For (locally dependent) $r$-BP on the $d$-dimensional torus $\mathbb{T}_L^d=(V,E)$ and two nodes $v,u\in V$, let $Pr[v\stackrel{p,r}{\longleftrightarrow} u\ \textrm{in}\ \mathbb{T}_L^d]$ be the probability that there is a path between $v$ and $u$ along the black nodes in the final configuration. We define $m_{-}(d,r,p):=\exp^{(r-2)}(\beta(d,r)p^{-\frac{1}{d-r+1}})$, where $\exp^{(r)}(x)=\exp(\exp^{(r-1)}(x))$ and $\exp^{(0)}(x)=x$.

\begin{lemma} (derived from Lemma 5.2 in~\cite{cerf2002threshold})
\label{lemma 6}
In locally dependent $r$-BP for $3\le r\le d$: there exist $\beta(d,r)>0$, $\gamma(d,r)>0$, $p(d,r)>0$ such that $\forall p<p(d,r)$ and $\forall m\le m_{-}(d,r,p)$, we have $Pr[v\stackrel{p,r}{\longleftrightarrow} u\ \textrm{in}\ \mathbb{T}_m^d]\le p^{\gamma\|v-u\|_{\infty}}$,
where $\|.\|_{\infty}$ denotes the infinity norm and $v,u$ are two nodes in $\mathbb{T}_m^d$.
\end{lemma}
Since the proof of Lemma 5.2 in~\cite{cerf2002threshold} is quite long, we do not reproduce the whole proof here. Instead, we point out how it should be changed in certain parts, where the independence in the initial configuration is used.

Consider (locally dependent) $r$-BP on $\mathbb{T}_m^d=(V,E)$, where $V=\{v_1,\cdots,v_{m^d}\}$. Define Bernoulli random variable $x_k$ for $1\le k \le m^d$ to be 1 if and only if node $v_k$ is white in the initial configuration. Since a single black node in the initial configuration suffices to make the whole torus black for $r=1$, we have
\[
Pr[v\stackrel{p,1}{\longleftrightarrow} u\ \textrm{in}\ \mathbb{T}_m^d]=1-Pr[\bigwedge_{k=1}^{m^d}x_k=1].
\]
In $r$-BP, this probability is equal to $1-(1-p)^{m^d}$ because each node is white independently with probability $1-p$. Since locally dependent $r$-BP does not enjoy the independence in the initial configuration, we cannot apply the same argument. However, we have
\[
Pr[v\stackrel{p,1}{\longleftrightarrow} u\ \textrm{in}\ \mathbb{T}_m^d]=1-Pr[\bigwedge_{k=1}^{m^d}x_k=1]=1-\prod_{k=1}^{m^d}Pr[x_k=1|\bigwedge_{k'=1}^{k-1}x_{k'}=1].
\]
We know that the whiteness of different nodes are positively correlated; that is, the probability of a node $v$ being white in the initial configuration does not decrease if we know that some other nodes are white in the initial configuration. Therefore, $Pr[x_k=1|\bigwedge_{k'=1}^{k-1}x_{k'}=1]\ge Pr[x_k=1]$. Since each node is black initially with probability $p$, we get $Pr[v\stackrel{p,1}{\longleftrightarrow} u\ \textrm{in}\ \mathbb{T}_m^d]\le 1-(1-p)^{m^d}$. As we will see later, this upper bound is all we need.

Then, they consider the case of $r=2$. They prove that in $r$-BP there exist $\beta(d,2)>0$, $C>0$ and $p(d,2)>0$ such that $\forall p<p(d,2)$ and $\forall m<m_{-}(d,2,p)$ the probability $Pr[v\stackrel{p,2}{\longleftrightarrow} u\ \textrm{in}\ \mathbb{T}_m^d]$ is at most $(C\|v-u\|_{\infty}^{d-1}p)^{\|v-u\|_{\infty}/2}$. The idea of the proof is as follows. Consider an integer $m\le m_{-}(d,2,p)=\beta(d,2)p^{-\frac{1}{d-1}}$. Let $v$ be a $d$-dimensional vector; we denote by $\underline{v}$ its first $d-1$ coordinates and by $\overline{v}$ the last one and write $v=(\underline{v},\overline{v})$. By symmetry, one can assume that $(\underline{v},\overline{v})$ and $(\underline{u},\overline{u})$ are such that $\overline{u}-\overline{v}=\|(\underline{v},\overline{v})-(\underline{u},\overline{u})\|_{\infty}$. Consider the \emph{slices}
\[
T_i:=\{(\underline{v},\overline{v})\in \mathbb{T}_m^d:\overline{v}\in\{2i,2i+1\}\}\ \textrm{for}\ i\in \mathbb{Z}.
\]
Suppose that there is a path along black nodes from $v$ to $u$ in the final configuration. Let $\mathscr{C}$ be the maximal connected set of black nodes in the final configuration which include $v$ and $u$. Let $a$ and $b$ be the first and the last indices of the slices intersecting $\mathscr{C}$. In all the slices $T_i$ for $i\in [a,b]$ there exists at least one black node $(\underline{w},\overline{w})$ such that $\|\underline{v}-\underline{w}\|_{\infty}\le \|v-u||_{\infty}$. The probability of this to happen in one fixed slice in $r$-BP is less than $1-(1-q)^{(2\|v-u\|_{\infty}+1)^{d-1}}$ where $q=2p-p^2$.  (Here, the estimate is similar to the $r=1$ estimate from above.) Furthermore, the slices being independent, one gets
\begin{equation}
\label{eq 1}
Pr[v\stackrel{p,1}{\longleftrightarrow} u\ \textrm{in}\ \mathbb{T}_m^d]\le d(1-(1-p)^{2(2\|v-u\|_{\infty}+1)^{d-1}})^{\|v-u\|_{\infty}/2}
\end{equation}
where the factor $d$ comes from the possible directions where $\|v-u\|_{\infty}$ is realized and we used $1-q=1-2p+p^2=(1-p)^2$. (Let us mention that the probability $q$ is not necessarily equal to $2p-p^2$ in locally dependent $r$-BP, but it is bounded by $p$ and $2p$; we will use this fact later.)

In locally dependent $r$-BP by applying our argument from above for $r=1$, we have that the probability that there exists at least one black node $(\underline{w},\overline{w})$ such that $\|\underline{v}-\underline{w}\|_{\infty}\le \|v-u||_{\infty}$ in one fixed slice is less than $1-(1-p)^{2(2\|v-u\|_{\infty}+1)^{d-1}}$. In contrast to $r$-BP, locally independent $r$-BP does not enjoy the independence of the slices, but we can consider $\|v-u\|_{\infty}/\alpha_1$ slices which are independent for some constant $\alpha_1>0$. Therefore, in locally dependent $r$-BP, we get
\begin{equation}
\label{eq 2}
Pr[v\stackrel{p,1}{\longleftrightarrow} u\ \textrm{in}\ \mathbb{T}_m^d]\le d(1-(1-p)^{{2(2\|v-u\|_{\infty}+1)^{d-1}}})^{\|v-u\|_{\infty}/\alpha_1}.
\end{equation}
Cerf and Manzo show that the right hand side of Equation~(\ref{eq 1}) can be upper-bounded by $(C\|v-u\|_{\infty}^{d-1}p)^{\|v-u\|_{\infty}/2}$ for some constant $C>0$. Applying basically the same calculations on the right hand side of Equation~(\ref{eq 2}) yields a similar upper bound in locally dependent $r$-BP. Using the estimate $1-\exp(x)\le -x$ implies that
\[
Pr[v\stackrel{p,1}{\longleftrightarrow} u\ \textrm{in}\ \mathbb{T}_m^d]\le d(-2(2\|v-u\|_{\infty}+1)^{d-1}\ln (1-p))^{\|v-u\|_{\infty}/\alpha_1}.
\]
For $p$ small, $\ln (1-p)\ge -2p$ and hence we have
\[
Pr[v\stackrel{p,2}{\longleftrightarrow} u\ \textrm{in}\ \mathbb{T}_m^d]\le (C\|v-u\|_{\infty}^{d-1}p)^{\|v-u\|_{\infty}/\alpha_1}.
\]
Therefore, in locally dependent $r$-BP we get the same upper bound as $r$-BP except that 2 is replaced by $\alpha_1$. We will see that this is all we need to prove our statement.

For $r\ge 3$, they apply an induction on the dimension $d$ and on the parameter $r$. First, they modify the initial configuration by adding some black nodes and they assume that some nodes become black if they have at least $r-1$ black neighbors instead of $r$. Such assumptions can be made due to the monotonicity of the process. They decompose the event $\{v\stackrel{p,r}{\longleftrightarrow} u\ \textrm{in}\ \mathbb{T}_m^d\}$. This upper-bounds the probability $Pr[v\stackrel{p,r}{\longleftrightarrow} u\ \textrm{in}\ \mathbb{T}_m^d]$ by the sum of the probability of some events of the following form: a particular set of slices must be fully black but not the slices in between, and the fully black slices are connected by some paths along black nodes. To compute the probability of such events, they utilize the independence of the slices. In locally dependent $r$-BP the slices are not independent, but we can consider a constant fraction of the slices which must be fully black such that they are independent, i.e., they are in distance at least $2T$ from each other. This is still problematic since for three selected slices $T_{i}$, $T_{i'}$, and $T_{i^{\prime\prime}}$ the event that there is a black path connecting $T_i$ to $T_{i'}$ and the event that there is a black path connecting $T_{i'}$ and $T_{i^{\prime\prime}}$ are not independent. To deal with this issue, we only consider the events for the connecting black paths one by one. This changes the exponent of our desired probability by a constant factor, similar to the case of $r=2$. By following their calculations, one can see that the effect of these constant factors appears in the choice of constant $\gamma$ in the satement of Lemma~\ref{lemma 6}. That is, the inequality $Pr[v\stackrel{p,r}{\longleftrightarrow} u\ \textrm{in}\ \mathbb{T}_m^d]\le p^{\gamma\|v-u\|_{\infty}}$ holds for a smaller value of $\gamma$.

Due to several lengthy calculations, we do not reproduce the whole proof. Basically, there are two small changes which one has to do to make the proof work for the locally dependent variant. 
\begin{itemize}
\item Firstly, at the end of page 80 they apply the result for $r=2$ which we discussed above, see Equation~(\ref{eq 1}). By some simplifications, they reach an upper bound of form $C_2q$, where $C_2>0$ is a constant and $q=p^2-2p$. By applying Equation~(\ref{eq 2}) instead of Equation~(\ref{eq 1}) and using the fact that $q$ is in the same magnitude as $p$, we get an upper bound of form $C_2'p$. (One needs to split the sum at the end of page 80 from an integer larger than 9.) 
\item As we discussed above, to get rid of the dependency among the slices, we choose a constant fraction of them. Therefore, in the calculations at the end of page 81, instead of $k$ we have $k/\alpha_2$ for some constant $\alpha_2>0$.
\end{itemize}
Both aforementioned constant factors can be hidden in constant $\gamma$ in the last line of calculations in page 81.

\subsection{Proof of Lemma~\ref{lemma 5}}
\label{proof of lemma 5}
\textbf{Lemma~\ref{lemma 5}.} \textit{In two-way $r$-BP on $\mathbb{T}_L^d$, if an $r$-dimensional hyper-square has occupied neighbor in at least $r$ distinct dimensions, it becomes occupied in $t'$ rounds for some even constant $t'$.}
\begin{proof}
Let us first set up some definitions. For an $r$-dimensional hyper-square $HS$ starting from $i=(i_1,\cdots,i_d)$, the set of nodes in $HS$ whose $j$-th coordinate is equal to $i_j$ (similarly $i_j+1$) induce an $(r-1)$-dimensional hyper-square, which is called a \emph{face} of $HS$; specifically the $(r-1)$-dimensional hyper-squares attained by fixing the $r$-th coordinate to be $i_r$ and $i_{r}+1$ are respectively called the \emph{upper face} and \emph{lower face}. Furthermore, the two outer neighbors of $HS$ in the $j$-th coordinate, where by definition $j$ is among the last $d-r$ coordinates, are simply attained by increasing or decreasing the $j$-th coordinate of all nodes in $HS$ by one. This implies that if the even-part (odd-part) of one of the outer neighbors of $HS$, say $HS'$, is black in some configuration, then each node in the even-part (resp. odd-part) of $HS$ has a black neighbor in $HS'$ in that configuration. In other words, if $HS'$ is occupied, then the upper face (similarly lower face) of $HS$ has an occupied neighbor in the $j$-th dimension. Similarly, if an inner neighbor of $HS$ in the $j$-th dimension for $1\leq j< r$, say $HS^{\prime\prime}$, is occupied, then the upper face of $HS$ (similarly lower face) has one occupied neighbor, namely the upper face (resp. lower face) of $HS^{\prime\prime}$. However, this is not the case for $j=r$. One of the inner neighbors in the $r$-th dimension has its upper face adjacent to the lower face of $HS$ (which implies if this neighbor is occupied, then the lower face of $HS$ has an occupied neighbor in the $r$-th dimension) and the other one has its lower face adjacent to the upper face of $HS$ (which provides an occupied neighbor in the $r$-th dimension for the upper face of $HS$ if it is occupied).

We prove our claim by induction on $r$. As the base case, we prove that for $r=2$ the statement is correct. Recall that we look at only the even rounds, otherwise we mention explicitly. Now, let $HS$ be a two-dimensional hyper-square starting from node $i$ with even parity (the odd case is analogous), then to become occupied it needs its even-part to become fully black. We want to show if $HS$ has occupied neighbors in two distinct dimensions in some configuration $\mathcal{C}_t$, it will be occupied in $\mathcal{C}_{t+t'}$ for some even constant $t'$. It has 4 inner neighbors and $2d-4$ outer neighbors. If two of the outer neighbors are occupied in $\mathcal{C}_t$, their odd-part must be black because their parity is different with $HS$. Thus, each node in the odd-part of $HS$ has two black neighbors, which implies that the odd-part becomes fully black in $\mathcal{C}_{t+1}$ and thus the even-part becomes black in $\mathcal{C}_{t+2}$, i.e., $HS$ is occupied. For the case that $HS$ has two occupied inner neighbors or one inner and one outer neighbor, see Figure~\ref{fig 4}.

Assume as the induction hypothesis that the claim is correct for $r-1\geq 2$, we prove it is true also for $r$. Suppose that the $r$-dimensional hyper-square $HS$ starting from $i$ has occupied neighbors in $r$ distinct dimensions in some configuration $\mathcal{C}_t$. Furthermore, assume the parity of $HS$ is even (the odd case is handled analogously). The lower face or upper face of $HS$ must have $r$ occupied neighbors as we discussed above; let it be the lower face. We claim one of the neighbors provides for each node in the even-part (similarly each node in the odd-part) of the lower face in every odd round (resp. even round) a black neighbor. To show that let us distinguish two cases. If one of the neighbors of $HS$ is outer, say the hyper-square $HS'$, then the lower face of $HS'$, which is a neighbor of the lower face of $HS$, satisfies our requirement. If there is no outer neighbor, then $HS$ has at least one occupied neighbor in each of the first $r$ dimensions. The neighbor in the $r$-th dimension which has its upper face adjacent to the lower face of $HS$ must be occupied (we assumed the lower face has $r$ occupied neighbors), which is then our required neighbor. Note that to occupy the lower face only making even nodes (the nodes in the even-part) black in even rounds or odd nodes (the nodes in the odd-part) black in odd rounds help because if for example the odd-part of the lower face is fully black in an even round, it does not have any impact on its occupation. By applying the induction hypothesis and the fact that one of the neighbors provides for each even node (similarly odd node) in each odd round (resp. even round) a black neighbor, we can conclude that the lower face must become occupied in a constant and even number of rounds. This is true because the remaining $r-1$ neighbors must make the lower face occupied under two-way $(r-1)$-BP and the extra neighbor needed by $r$-BP is always provided. Now, we can apply the same argument on the upper face by setting the lower face as the neighbor which provides for each even node (similarly odd node) in the upper face in each even round (resp. odd round) a black neighbor.   
\end{proof}
\end{document}